\documentclass[12pt]{amsart}
\usepackage{mathrsfs}
\usepackage{pgf}
\usepackage{tikz}
\usetikzlibrary{calc,through,backgrounds}
\usepackage{amssymb}
\usepackage{amsmath}
\usepackage{amsthm}
\newcommand{\NN}{\mathbb{N}}
\newcommand{\ZZ}{\mathbb{Z}}

\newcommand{\Mod}{{\rm \,\,\,mod\,\,\,}}
\newcommand{\dd}{{\mathrm d}}
\newcommand{\supp}{{\rm supp}}
\newcommand{\fact}{{\mathsf Z}}
\newcommand{\len}{{\mathsf L}}

\newcommand{\atoms}{\mathcal{A}}
\newcommand{\cat}{{\rm c}}
\newcommand{\tame}{{\rm t}}
\newcommand{\Ap}{{\rm Ap}}

\theoremstyle{plain}
\newtheorem{theorem}{Theorem}[section]
\newtheorem{lemma}[theorem]{Lemma}
\newtheorem{proposition}[theorem]{Proposition}
\newtheorem{corollary}[theorem]{Corollary}

\theoremstyle{definition}

\theoremstyle{remark}
\newtheorem{remark}[theorem]{Remark}
\newtheorem*{example}{Example}
\newtheorem{notation}[theorem]{Notation}

\title[The catenary and tame degree of numerical monoids]{The catenary and tame degree of numerical monoids generated by generalized arithmetic sequences}
\author{M. Omidali}
\address{Department of Mathematics, Bu-Ali Sina University, Hamedan, Iran}
\email{mehdioa@gmail.com}
\begin{document}
\maketitle

\begin{abstract}
Studying ceratin combinatorial properties of non-unique factorizations
have been a subject of recent literatures. Little is known about
two combinatorial invariants, namely the catenary degree and the tame degree, even in the
case of numerical monoids. In this paper we compute these invariants for a certain
class of numerical monoids generated by generalized arithmetic sequences. We also
show that the difference between the tame degree and the catenary degree can be arbitrary
large even if the number of minimal generators is fixed.
\end{abstract}

\textbf{Keywords.}
Numerical monoids, Tame degree, Catenary degree

\textbf{2010 Mathematics Subject Classification.} 
20M13
\section{Introduction}
Various aspects of non-unique factorizations in integral domains have
been a subject of researchers in recent years. There are several arithmetic invariants that
measure the behavior of non-unique factorizations in integral domains.
Early works on the behavior of non-unique factorizations in integral domains
were focused on the length of irreducible factorizations of an element. Delta sets
and the elasticity are some of these invariants that measure how far an integral domain
or monoid is from being factorial or half-factorial (i.e., all the irreducible factorizations
of an element have the same length). In recent years other invariants which have close
relations with the distance between irreducible factorizations have been appearing in literatures.
The \textit{catenary degree} and the \textit{tame degree} are such invariants.
The monograph of Geroldinger and Halter-Koch
\cite{GH}, the recent survey \cite{GHservey} and \cite{GGS} are good references for studying the catenary degree and the tame degree
in commutative cancellative monoids.

Throughout $\NN$ will be the set of non-negative integers. A
numerical monoid is a submonoid of $\NN$ (that is closed under
addition) which contains $0$ and its complement in $\NN$ is finite. Every numerical monoid is
necessarily finitely generated and has a minimal set of generators.
For a set $\{n_1,\ldots,n_p\}$ of increasing positive integers,
the numerical monoid generated by $n_i$'s is $\langle n_1,\ldots,n_p\rangle:=\{\sum_{i=1}^pz_in_i|z_i\in\NN,\,\,i=1,\ldots,p\}$.
Delta sets of numerical monoids were studied in \cite{ACHP,BCKR}, and it is shown that in
numerical monoids generated by arithmetic sequences, Delta sets reduce to a single
element. This element is the difference between two consecutive elements of the arithmetic
sequence. Also in \cite{CHM} it is shown the elasticity of a numerical monoid equals
the quotient of the largest by the smallest minimal generator. For numerical monoids
generated by an arithmetic sequence the catenary degree and the tame degree were determined in \cite{CGL}.

In this paper we consider numerical monoids generated by generalized arithmetic sequences, that is
numerical monoids of the form $S=\langle a,ha+d,\ldots,ha+xd\rangle$, where $a,h,x$, and $d$ are
positive integers and $\gcd(a,d)=1$. This class of numerical monoids contains all numerical monoids
generated by arithmetic sequences. We explicitly compute their catenary and tame degrees. As a
result we show that the difference between the tame degree and the catenary degree can be arbitrary large
even if the number of minimal generators is fixed. From [\cite{GH}, Example 3.1.6] it is known that
$$\cat(S)\leq \tame(S)\leq \frac{g(S)+n_p}{n_1}+1,$$
where $\cat(S)$ and $\tame(S)$ are the catenary degree and the tame degree of $S$ and $g(S)$ is the Frobenius number of $S$ (i.e. the largest positive integer not belonging to $S$). For numerical monoids generated
by generalized arithmetic sequences we show that the difference between $\cat(S)$ and $(g(S)+n_p)/n_1+1$ can
be arbitrarily large. Moreover,  $\tame(S)$ can be $\cat(S)$ or $\lfloor (g(S)+n_p)/n_1+1\rfloor$ and all of these things can happen even if the number of minimal generators is fixed.

For general theory of numerical monoids we refer the reader to \cite{RG}. Other aspects of numerical monoids generated by generalized arithmetic sequences are studied in \cite{ACKT}, \cite{Mat},
\cite{OR}, \cite{Rod}, \cite{Sel}. We mainly use a membership criterion presented in \cite{OR}. Computing some combinatorial invariants of numerical monoids, or more generally commutative cancellative
monoids such as the elesticity and the tame degree requires us to determine the set of irreducible elements
of a congruence related to the monoid \cite{CGLR}. All the results in this paper have been tested with the
Numericalsgps package of GAP \cite{SGPS}.

\begin{notation} For a rational number $r$ by $\lceil
r \rceil$ we mean the least integer bigger than or equal to $r$ and
by $\lfloor r \rfloor$ we mean the greatest integer less than or
equal to $r$. Also $(m \mod n)$ means the reminder of quotient of
the integer division of $m$ by $n$ where $m,n\in \ZZ$. For integers
$a,b,$ and $c$ by $a\equiv b(\bmod c)$ we mean that $a-b$ is
divisible by $c$.
\end{notation}
\section{Preliminaries}

Suppose that $S$ is a numerical monoid which is minimally
generated by $\{n_1,\ldots,n_p \}$. Each $n_i$ is called an atom and $\{n_1,\ldots,n_p\}$ is called the set of atoms of $S$ and is denoted by $\atoms(S)$. Define a partial order $\leq_S$ on $S$ by declaring $a\leq_S b$ if $b-a \in S$. The factorization morphism of
$S$ is
$$\varphi:\NN^p\longrightarrow S,\,\,\,
\varphi(z_1,\ldots,z_p)=z_1n_1+\ldots+z_pn_p.$$ Then $S$ is
isomorphic to $\NN^p/\sigma$, where $a\sigma b$ if
$\varphi(a)=\varphi(b)$. The set of factorizations of $n\in S$ is
defined by
$$\fact(n)=\varphi^{-1}(n)=\{(z_1,\ldots,z_p)\in\NN^p|
z_1n_1+\ldots+z_pn_p=n\}.$$

If $z=(z_1,\ldots,z_p)$ is a factorization of $n$, then the length
of $z$ is $|z|=z_1+\ldots+z_p$. The set of lengths of $n$ is defined by
\[\len(n)=\{|z|\,\big|\, z\in\fact(n)\}\subset\NN_0.\]
Also the support of $z$ is defined by
$$\supp(z)=\{i\in\{1,\ldots,p\}|z_i\neq 0\}.$$
For $z=(z_1,\ldots,z_p),z'=(z'_1,\ldots,z'_p)\in\NN^p$ we set
$$\gcd(z,z')=(\min\{z_1,z'_1\},\ldots,\min\{z_p,z'_p\}),
\frac{z}{z'}=z-z'.$$
The distance between $z$ and $z'$ is defined by
$$\dd(z,z')=\max\Big\{\Big|\frac{z}{\gcd(z,z')}\Big|,\Big|\frac{z'}{\gcd(z,z')}\Big|\Big\}.$$
For two nonempty subsets $X$ and $Y$ of $\NN^p$ we define the distance between $X$ and $Y$ as follows
$$\dd(X,Y)=\min\{\dd(x,y)| x\in X, y\in Y\}.$$
If $X=\{x\}$ then we use $\dd(x,Y)$ instead of $\dd(X,Y)$.

Let $n\in S$ and $z,z'\in \fact(n)$. Then an $N-$\textit{chain of factorizations} from $z$
to $z'$ is a sequence $z_0,\ldots,z_k\in \fact(n)$ such that $z=z_0$, $z'=z_k$ and
$\dd(z_i,z_{i+1})\leq N$ for all $i$. The catenary degree of $n$, $\cat(n)$, is defined as the
least non-negative integer $N$ such that for any two factorizations $z,z'\in \fact(n)$, there
is an $N$-chain from $z$ to $z'$. The catenary
degree of $S$ is defined by
\[\cat(S)=\sup\{\cat(n)|n\in S\}.\]

Two elements $z$ and $z'$ of $\NN^p$ are $\mathscr{R}$-related if there exists a sequence
$z=z_0,\ldots,z_k=z'$ in $\NN^p$ such that $\supp(z_i)\cap\supp(z_{i+1})\neq\varnothing$, for every
$i\in\{1,\ldots,k-1\}$. As the number of different factorizations of an element $n\in S$ is finite, so
is the number of different $\mathscr{R}$-classes of $\fact(n)$. Let $n\in S$ and let $\mathscr{R}_1^n,\ldots,
\mathscr{R}^n_{k_n}$ be the different $\mathscr{R}$-classes of elements in $\fact(n)$. Set $\mu(n)=\max\{r_1^n,\ldots,r_{k_n}^n\}$, where $r^n_i=\min\{|z|\big| z\in\mathscr{R}^n_i\}$.
Define
$$\mu(S)=\max\{\mu(n)\big|n\in S \,\,\mathrm{and}\,\, k_n\geq 2\}.$$
\begin{theorem}[\cite{CGLPR}, Theorem 3.1] Let $S$ be a numerical monoid. Then
$$\cat(S)=\mu(S).$$
\end{theorem}

Let $n\in S$. For $i\in\{1,\ldots,p\}$ we set $\fact^i(n)=\{(z_1,\ldots,z_p)\in \fact(n)| z_i\neq 0\}$.
Suppose that $n-n_i\in S$ for some $i\in\{1,\ldots,p\}$. We define
\[\tame_i(n)=\max\{\dd(z,\fact^i(n))|z\in \fact(n)\}.\]
Setting
\[\tame(n)=\max\{\tame_i(n)|n-n_i\in S, 1\leq i\leq p\},\]
and 
\[\tame_i(S)=\max\{\tame_i(n)|n-n_i\in S\}, \quad i=1,2,\ldots,p,\]
we define the tame degree of $S$ by
\[\tame(S)=\max\{\tame(n)|n\in S\}=\max\{\tame_i(S)|i=1,2,\ldots,p\}.\]
It is clear from definitions that $\cat(n)\leq \tame(n)\leq \max\len(n)$.

For $a, b\in S$ let $\omega(a,b)$ be the smallest $N\in\NN_0\cup\{\infty\}$ with the following property:
\begin{itemize}
\item For all $n\in\NN$ and $a_1,\ldots,a_n\in S$, if $a=a_1+\ldots+a_n$ and $b\leq_S a$ then there exists a subset $\Omega\subset [1,n]$ such that $|\Omega|\leq N$ and
    \[b\leq_S\sum_{\nu \in \Omega}a_\nu.\]
\end{itemize}
Note that if $b\nleq_S a$, then set $\omega(a,b)=0$. For $b\in S$ we define
\[\omega(S,b)=\sup\{\omega(a,b)\,\mid a\in S\}\in\NN_0\cup\{\infty\}.\]

For $k\in\NN$ and $b\in S$ we set
\begin{align*}
\tau_k(S,b)=&\sup\{\min\len(a-b)\,\big|\, a=u_1+\ldots+u_j\in b+S\,\text{with}\, j\in [0,k],\,\\
& u_1,\ldots,u_j\in\atoms(S),\text{and}\, b\nleq_S a-u_i\, \text{for all}\, i\in[1,j]\}\in\NN_0\cup\{\infty\}
\end{align*}
and
\[\tau(S,b)=\sup\{\tau_k(S,b)\,\big|\, k\in\NN\}\in\NN_0\cup\{\infty\}.\]
We define the $\omega$ invariant and the $\tau$ invariant of $S$ as
\[\omega(S)=\sup\{\omega(S,n_i)\,\mid 1\leq i\leq p\}\in\NN_0\cup\{\infty\},\]
\[\tau(S)=\sup\{\tau(S,n_i)\,\mid  1\leq n_i\leq p\}\in\NN_0\cup\{\infty\}.\]

\begin{theorem} If $S\neq \NN$, then for every $1\leq i\leq p$ we have
\[\tame_i(S)=\max\{\omega(S,n_i),1+\tau(S,n_i)\}\in\NN_{\geq 2}\cup\{\infty\}.\]
\end{theorem}
\begin{proof}
Since $S\neq \NN$, all $n_i$ are atoms but not primes. Thus the assertion follows from [\cite{GH08}, Theorem 3.6]. 
\end{proof}
\section{Numerical monoid generated by generalized arithmetic sequences}
Let $S=\langle a,ha+d,\ldots,ha+xd\rangle$ be a numerical monoid where $a,d,h$, and $x$
are positive integers and $\gcd(a,d)=1$. If $i\geq a$ then
$ha+id=\lfloor \frac{i}{a}\rfloor da+(ha+(i \mod a)d) \in S$. Therefore we may assume that $x<a$.
In this case $\{a,ha+d,\ldots,ha+xd\}$ is the minimal set of generators of $S$ (\cite{OR}, Corollary 2.5). We sometimes
use $n_0,n_1,\ldots,n_x$ instead of $a,ha+d,\ldots,ha+xd$, respectively.

\begin{theorem}[\cite{OR}, Proposition 2.1] Suppose that $S=\langle a,ha+d,\ldots,ha+xd\rangle$ where $a,d,h$, and $x$
are positive integers, $\gcd(a,d)=1$, and $1\leq x\leq a-1$. Let $n=qa+id$ where $q,i\in\NN$
and $0\leq i\leq a-1$. Then $n\in S$ if and only if $\lceil\frac{i}{x}\rceil h\leq q$.
\end{theorem}
Note that by a generalization of the euclidean algorithm every integer $n$ has a unique
representation of the form $n=qa+id$ with $0\leq i\leq a-1$. So the above condition gives a simple
membership criterion for numerical monoids generated by generalized arithmetic sequences.

Let $S$ be a numerical monoid and $n\in S$. The Ap\'{e}ry set
associated to $n$ is defined as
$$\Ap (S,n)=\{s\in S| s-n\notin S\}.$$
It is easy to see that $\Ap (S,n)$ has $n$ elements of different
congruence classes modulo $n$. Therefore each element of $S$ can be
uniquely expressed as $kn+w$, where $k\in\NN$ and $w\in \Ap (S,n)$.
\begin{theorem}[\cite{OR}, Proposition 2.6]\label{apery}Let $S=\langle a,ha+d,\ldots,ha+xd\rangle$ where $a,d,h$, and $x$
are positive integers, $\gcd(a,d)=1$, and $1\leq x\leq a-1$. Then
$$\Ap(S,a)=\{\lceil\frac{i}{x}\rceil ha+id| 0\leq i\leq a-1\}.$$
\end{theorem}

Let $S$ be a numerical monoid generated minimally by $n_1,n_2,\ldots,n_p$. For every $n\in S$ we define
$$C_n=\{V\subseteq \{1,2,\ldots,p\}| \exists z\in\fact(n), V\subseteq \supp(z)\}.$$
It is easy to see that $C_n$ is a simplicial complex and the number of $\mathscr R$-classes of $n$ is equal to the number of path components of $C_n$. We define a graph $G_n=(V_n,E_n)$ as 
\begin{align*}
&V_n=\{i\in\{1,2,\ldots,p\}| n-n_i\in S\},\\
&E_n=\{\{i,j\}|n-(n_i+n+j)\in S, i\neq j\in\{1,2,\ldots,p\}\},
\end{align*}
with $V_n$ and $E_n$ are the sets of vertices and edges of $G_n$, respectively.
Then the number of path components of $C_n$ and $G_n$ are equal and thus  the number of $\mathscr R$-classes of $n$ is equal to the number of path components of $G_n$.
\begin{example} Let $S=\langle 5,6,7,9\rangle$ and $n=18$. Then $18$ can be written in terms of minimal generators as follows:
$$18=2\times 9=5+6+7.$$
$C_{18}$ and $G_{18}$ are depicted in Figure~\ref{CandG}.

\begin{figure}[h]
\begin{tikzpicture}[thick,help lines/.style={thin,draw=black}]
\def\A{$1$} \def\B{$2$}
\def\C{$3$} \def\D{$4$}
\colorlet{input}{black} \colorlet{output}{black}
\colorlet{triangle}{gray}
\coordinate [label=left:\A] (A) at ($ (3,0)$);
\coordinate [label=left:\B] (B) at ($ (3,-1) $);
\coordinate [label=right:\C] (C) at ($ (4,-1) $);
\coordinate [label=right:\D] (D) at ($ (4,0) $);
\draw [input] (A) -- (B) -- (C) -- (A);

\foreach \point in {A,B,C,D}
\fill [black] (\point) circle (2pt);
\begin{pgfonlayer}{background}
\fill[triangle!80] (A) -- (C) -- (B) -- cycle;
\end{pgfonlayer}
\node [below right,text justified] at (3.1,-1.5) {$C_{18}$
};

\def\a{$1$} \def\b{$2$}
\def\c{$3$} \def\d{$4$}
\colorlet{input}{black} \colorlet{output}{black}
\colorlet{triangle}{gray}
\coordinate [label=left:\a] (a) at ($ (8,0)$);
\coordinate [label=left:\b] (b) at ($ (8,-1) $);
\coordinate [label=right:\c] (c) at ($ (9,-1) $);
\coordinate [label=right:\d] (d) at ($ (9,0) $);
\draw [input] (a) -- (b) -- (c) -- (a);

\foreach \point in {a,b,c,d}
\fill [black] (\point) circle (2pt);
\begin{pgfonlayer}{background}
\end{pgfonlayer}
\node [below right,text justified] at (8.1,-1.5) {$G_{18}$
};
\end{tikzpicture}
\caption{$C_n$ and $G_n$ for $S=\langle 5,6,7,9\rangle$ and $n=18$}\label{CandG}
\end{figure}
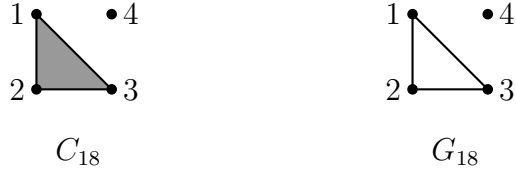
\end{example}

In the proof of [\cite{OR}, Theorem 2.16] we found all elements $n$ of a numerical monoid generated by a generalized
arithmetic sequence for which $G_n$ is disconnected, and
we summarize it in the following Theorem.
\begin{theorem}\label{minpres} Suppose that $S=\langle a,ha+d,\ldots,ha+xd\rangle$ where $a,d,h$, and $x$
are positive integers, $\gcd(a,d)=1$, and $1\leq x\leq a-1$. Let $n\in S$ be such that $C_n$ is disconnected. Then
either $n=n_i+n_j$, with $1\leq i,j\leq x$, or $n\in\{pn_x+n_{r+1},\ldots,pn_x+n_x=(p+1)n_x\}$, where
$p=\lfloor\frac{a-1}{x}\rfloor$ and $r=(a-1)\Mod x$.
\end{theorem}
\begin{proposition} Suppose that $S=\langle a,ha+d,\ldots,ha+xd\rangle$ where $a,d,h$, and $x$
are positive integers, $\gcd(a,d)=1$, and $1\leq x\leq a-1$. Let
$n=qa+id\in S$ be such that $0\leq i<a$. Then
$$\max\len(n)=q-\lceil\frac{i}{x}\rceil(h-1).$$
\end{proposition}
\begin{proof} First we show that $\max\len(n)\geq q-\lceil\frac{i}{x}\rceil(h-1)$. Let $i=\lceil\frac{i}{x}\rceil x-s$ with $0\leq s<x$.
\begin{itemize}
\item If $s=0$ then
$$n=qa+id=qa+\lceil\frac{i}{x}\rceil xd=(q-\lceil\frac{i}{x}\rceil
h)a+\lceil\frac{i}{x}\rceil(ha+xd).$$ Therefore
$$\max\len(n)\geq (q-\lceil\frac{i}{x}\rceil
h)+\lceil\frac{i}{x}\rceil=q-\lceil\frac{i}{x}\rceil(h-1).$$
\item If $s>0$ then $1\leq \lceil\frac{i}{x}\rceil$ and $\lceil\frac{i}{x}\rceil h\leq q$. Therefore
\[\begin{split}
n&=qa+id\\
&=qa+(\lceil\frac{i}{x}\rceil x-s)d\\
&=(q-\lceil\frac{i}{x}\rceil h)a+(\lceil\frac{i}{x}\rceil-1)(ha+xd)+(ha+(x-s)d).\\
\end{split}\]
Thus
$$\max\len(n)\geq(q-\lceil\frac{i}{x}\rceil h)+(\lceil\frac{i}{x}\rceil-1)+1=q-\lceil\frac{i}{x}\rceil(h-1).$$
\end{itemize}

Now we show that $\max\len(n)\leq q-\lceil\frac{i}{x}\rceil(h-1)$. We
use induction on $q$ to show this. If $q=0$ then
$\lceil\frac{i}{x}\rceil h\leq 0$. Therefore $i=0$, $n=0$, and
$0=\max\len(n)\leq q-\lceil\frac{i}{x}\rceil(h-1)$. Now let $q>0$. We
have
$$\max\len(n)=1+\max\{\max\len(n')|n'\in \{n-n_0,n-n_1,\ldots,n-n_x\}\cap S\}.$$
So it is enough to show that for any $0\leq j\leq x$, if
$n'=n-n_j\in S$, then $\max\len(n')+1\leq
q-\lceil\frac{i}{x}\rceil(h-1)$. There are three cases.
\begin{itemize}
\item First let $j=0$ and $n'=n-n_0=(q-1)a+id\in S$. Then by induction
$$\max\len(n')+1\leq
(q-1)-\lceil\frac{i}{x}\rceil(h-1)+1=q-\lceil\frac{i}{x}\rceil(h-1).$$
\item Now let $1\leq j\leq i$ and $n'=n-n_j=(q-h)a+(i-j)d\in S$. Then
$0\leq i-j<a$ and $\lceil\frac{i-j}{x}\rceil \geq
\lceil\frac{i-x}{x}\rceil= \lceil\frac{i}{x}\rceil-1$. By induction
we have
$$\max\len(n')+1\leq (q-h)-\lceil\frac{i-j}{x}\rceil(h-1)+1\leq
q-\lceil\frac{i}{x}\rceil(h-1).$$
\item Finally suppose that $i<j$ and $n'=n-n_j=(q-h)a+(i-j)d\in S$.
Then $n'=(q-h-d)a+(i-j+a)d$ with $0\leq i-j+a<a$. By induction and
the fact that $i-j+a\geq i$ we have
$$\max\len(n')+1\leq (q-h-d)-\lceil\frac{i-j+a}{x}\rceil(h-1)+1\leq q-\lceil\frac{i}{x}\rceil(h-1).$$
\end{itemize}
\end{proof}
\begin{corollary}\label{order}Suppose that $S=\langle a,ha+d,\ldots,ha+xd\rangle$ where $a,d,h$, and $x$
are positive integers, $\gcd(a,d)=1$, and $1\leq x\leq a-1$. Let $n=qa+id$ then $n\in S$
if and only if
\[\lceil\frac{i\Mod a}{x}\rceil h\leq q+\lfloor\frac{i}{a}\rfloor d,\]
and in this case
\[\max\len(n)=(q+\lfloor\frac{i}{a}\rfloor d)-\lceil\frac{i\Mod a}{x}\rceil(h-1).\]
\end{corollary}
\begin{proof} Just note that $n=(q+\lfloor\frac{i}{a}\rfloor d)a+(i\Mod a)d$.
\end{proof}
\subsection{The catenary degree}
We are ready to find the catenary degree of numerical monoids generated by
generalized arithmetic sequences.
\begin{theorem}[\cite{CGL}, Proposition 5] \label{mincatenary} Let $S$ be a numerical monoid minimally generated
by $\{n_1,\ldots,n_p\}$ with $0<n_1<\ldots<n_p$. Then
\begin{equation}\label{bound}\min\{k\in\NN\setminus\{0\}\,|\,kn_1\in\langle n_2,\ldots,n_p\rangle\}\leq \cat(S).\end{equation}
\end{theorem}
Even though there are numerical monoids for which the inequality (\ref{bound}) is strict, we will
show that it is an equality for numerical monoids generated by generalized arithmetic sequences.
Not all numerical monoids benefit from this nice property. It even fails for
numerical monoids generated by almost arithmetic sequences (i.e., numerical monoids generated by
a set for which all but one element form some consecutive elements of an arithmetic sequence). For
example any three relatively prime integers form an almost arithmetic sequence but
$S=\langle 6,9,11\rangle$ for which we have $\cat(S)=4$ but $\min\{k\in\NN\setminus\{0\}| 6k\in\langle 9,11\rangle\}=3$.
\begin{lemma}\label{mingeneral} Let $a,h,x$, and $d$ be positive integers with $x<a$ and $\gcd(a,d)=1$. Then
$$\min\{k\in\NN\setminus \{0\}|\, ka\in\langle
ha+d,\ldots,ha+xd\rangle\}=\lceil\frac{a}{x}\rceil h+d.$$
\end{lemma}
\begin{proof} Let $c=\gcd(ha+d,\ldots,ha+xd)=\gcd(h,d)$. Then
\[\begin{split}\min\{k\in\NN\setminus\{0\}\,|\, &ka\in\langle ha+d,\ldots,ha+xd\rangle\}\\
&=c\min\{k\in\NN\setminus\{0\}\,|\, ka\in\langle\frac{h}{c}a+\frac{d}{c},\ldots,\frac{h}{c}a+x\frac{d}{c}\rangle\}.\end{split}\]
Therefore we must show that
\[\min\{k\in\NN\setminus\{0\}\,|\, ka\in\langle\frac{h}{c}a+\frac{d}{c},\ldots,\frac{h}{c}a+x\frac{d}{c}\rangle\}=
\lceil\frac{a}{x}\rceil\frac{h}{c}+\frac{d}{c}.\]
So we may assume that $\gcd(h,d)=1$.
Let $a'=ha+d, h'=1, x'=x-1$ and $d'=d$. Then $S'=\langle a',h'a'+d',\ldots,h'a'+x'd'\rangle=\langle ha+d,\ldots,ha+xd\rangle$ is a numerical
monoid generated by a generalized arithmetic sequence. Let
$q'a'+i'd=ka$ with $0\leq i'<a'$ and $\lceil\frac{i'}{x'}\rceil h'\leq
q'$ and $k\neq 0$. $\lceil\frac{i'}{x'}\rceil h' \leq q'$ is equivalent
to $i'\leq q'x'$. Since $k\neq 0$ we have $q'>0$. From $q'a'+i'd=ka$
we have $(q'+i')d=(k-q'h)a$ and from $\gcd(a,d)=1$ we deduce that there is
$j>0$ such that $q'+i'=ja$ and $k-q'h=jd$. From $q'+i'=ja$ and
$i'\leq q'x'$ we deduce that $ja-q'\leq q'x'$ or equivalently
$ja\leq q'(x'+1)=q'x$. Thus, if $k>0$, then $ka\in S'$ if and only if there are
$q'>0$ and $j>0$ such that
$$ja\leq q'x,\,\,\text{and}\,\, k=q'h+jd.$$
The minimum of $k$ with the stated property is attained if we set
$j=1$. In this case $a\leq q'x$ and thus
$\lceil\frac{a}{x}\rceil\leq q'$. So the minimum of $k>0$ with the
property $ka\in S'$ is when $q'=\lceil\frac{a}{x}\rceil$ and $j=1$
which is $k=q'h+d=\lceil\frac{a}{x}\rceil h+d$.
\end{proof}
\begin{theorem}\label{cat} Let $S=\langle a,ha+d,\ldots,ha+xd\rangle$ where $a,d,h$, and $x$
are positive integers, $\gcd(a,d)=1$, and $1\leq x\leq a-1$. Then
\[\cat(S)=\lceil\frac{a}{x}\rceil h+d.\]
\end{theorem}
\begin{proof} We know that $\cat(S)\geq \lceil\frac{a}{x}\rceil h+d$ by Theorem \ref{mincatenary} and Lemma
\ref{mingeneral}. We complete the proof by
showing that if $n\in S$ is such that the number of $\mathscr{R}$-classes of $n$ is bigger than
one then $\max\len(n)\leq \lceil\frac{a}{x}\rceil h+d$. By Theorem \ref{minpres} we have two cases.
\begin{itemize}
\item First let $n=n_i+n_j=2ha+(i+j)d$ with $1\leq i,j\leq x$. Then
\[\begin{split}
\cat(n)\leq\max\len(n)&=2h+\lfloor\frac{i+j}{a}\rfloor d-\lceil\frac{(i+j)\Mod a}{x}\rceil(h-1)\\
&\leq 2h+\lfloor\frac{i+j}{a}\rfloor d\\
&\leq 2h+d\\
&\leq \lceil\frac{a}{x}\rceil h+d.
\end{split}\]
\item Now let $n\in\{pn_x+n_{r+1},\ldots,pn_x+n_x=(p+1)n_x\}$, with
$p=\lfloor\frac{a-1}{x}\rfloor$ and $r=(a-1)\Mod x$. If $r+1\leq j<x$ and
$n=pn_x+n_j$ then $n=(\lfloor\frac{a-1}{x}\rfloor+1)ha+jd=\lceil\frac{a}{x}\rceil ha+jd$ and
$$\max\len(n)=\lceil\frac{a}{x}\rceil h-\lceil\frac{j}{x}\rceil(h-1)\leq \lceil\frac{a}{x}\rceil h+d.$$
If $n=(p+1)n_x$ then $n=\lceil\frac{a}{x}\rceil ha+(px+x)d=(\lceil\frac{a}{x}\rceil h+d)a+(px+x-a)d=(\lceil\frac{a}{x}\rceil h+d)a+(x-1-r)d$. Since
$0\leq x-1-r< x$ we have
$$\max\len(n)=\lceil\frac{a}{x}\rceil h+d-\lceil\frac{x-1-r}{x}\rceil(h-1)\leq \lceil\frac{a}{x}\rceil h+d.$$
\end{itemize}
\end{proof}
\subsection{The tame degree} Now we find the tame degree of numerical monoids generated by
generalized arithmetic sequences.
\begin{theorem}[\cite{CGL}, Theorem 16]\label{tametest} Let $S$ be a numerical monoid minimally generated by $\{n_1<\ldots,n_p\}$. Let $n\in S$ be minimal satisfying $\tame(n)=\tame(S)$. Then $n=w+n_i$ for
some $i\in\{1,\ldots,p\}$ and $w\in \Ap(S,n_j)$ with $j\in\{1,\ldots,p\}\setminus \{i\}$.
\end{theorem}
\begin{theorem} Let $S=\langle a,ha+d,\ldots,ha+xd\rangle$ where $a,d,h$, and $x$
are positive integers, $\gcd(a,d)=1$, and $1\leq x\leq a-1$. Then
\[\tame(S)=\omega(S)=\big(\lceil\frac{a-1}{x}\rceil +1\big)h+d.\]
\end{theorem}
\begin{proof} We set $t=\lceil\frac{a-1}{x}\rceil h+h+d$ and $n=ta$. Then
$$w=n-n_1=(\lceil\frac{a-1}{x}\rceil h+d)a-d=\lceil\frac{a-1}{x}\rceil ha+(a-1)d\in\Ap(S,a)$$
Therefore there exists a factorization $z\in\fact^1(n)$. Also $z'=(t,0,\ldots,0)\in\fact(n)\setminus\fact^1(n)$ and
therefore $\tame(S)\geq\tame_1(n)\geq \dd(z',\fact^1(n))=|z'|=t$. We complete the proof by concluding that $\tame(S)\leq
t$.
Let $n\in S$ be minimal such that $\tame(n)=\tame(S)$. Assume that $z$ is a factorization of
$n$ such that $\dd(z,z')=\tame(n)$ with $z'\in\fact^i(n)$. From the minimality of $n$, we can derive
that $\supp(z)\cap\supp(z')=\varnothing$ and thus $\dd(z,z')=\max\{|z|,|z'|\}$. Let $j$ be in the
support of $z$. Theorem \ref{tametest} ensures that $n-(n_i+n_j)\notin S$. Thus $n=(n-n_i)+n_i=(n-n_j)+n_j$ with
$n-n_j\in\Ap(S,n_i)$ and $n-n_i\in\Ap(S,n_j)$.
\begin{itemize}
\item If either $i$ or $j$, say $i$, is $0$ then $n-n_j=\lceil\frac{k}{x}\rceil ha+kd$ for
some $0\leq k<a$ by Theorem \ref{apery}. Then
\[\begin{split}
\tame(n)&\leq\max\len(n)=\max\len\big((\lceil\frac{k}{x}\rceil h+h)a+(k+j)d\big)\\
&=\lceil\frac{k}{x}\rceil h+h+\lfloor\frac{k+j}{a}\rfloor d-\lceil\frac{(k+j)\Mod a}{x}\rceil (h-1)\\
&\leq \lceil\frac{k}{x}\rceil h+h+\lfloor\frac{k+j}{a}\rfloor d\\
&\leq \lceil\frac{a-1}{x}\rceil h+h+\lfloor\frac{k+j}{a}\rfloor d\\
&\leq \lceil\frac{a-1}{x}\rceil h+h+d=t.\\
\end{split}\]
\item Let neither $i$ nor $j$ is $0$. Assume that $\tame(S)>t$. Then either $z$ or $z'$ is greater
than $t$. This implies that either $n-n_i$ or $n-n_j$ has order greater than or equal to $t$. Let for
example $\max\len(n-n_i)\geq t$. Since $n-n_i-n_j\notin S$ we have
\[\lceil\frac{(k-i-j)\Mod a}{x}\rceil h>q-2h+\lfloor\frac{k-i-j}{a}\rfloor d\]
and therefore
\[\begin{split}
\max\len(n-n_i)&=\max\len((q-h)a+(k-i)d)\\
&=\max\len((q-h+\lfloor\frac{k-i}{a}\rfloor d)a+((k-i)\Mod a)d)\\
&=q-h+\lfloor\frac{k-i}{a}\rfloor d-\lceil\frac{(k-i)\Mod a}{x}\rceil(h-1)\\
&\leq q-h+d\\
&<\lceil\frac{(k-i-j)\Mod a}{x}\rceil h+h-\lfloor\frac{k-i-j}{a}\rfloor d+d\\
&<\lceil\frac{a-1}{x}\rceil h+h+d=t,\\
\end{split}\]
a contradiction.
\end{itemize}
Therefore $\tame(S)=t$.

We know that $\omega(S)\leq \tame(S)$. We complete the proof by showing that $\omega(n,n_1)\geq t$, where $n=ta$. We saw that $n_1\leq_S n$; again we consider the decomposition
$$n=\underbrace{a+\ldots+a}_{t-\text{times}}.$$
Then $n-n_1\in\Ap(S,n_0)$; so $n_1\nleq_S n-n_0$, $\omega(n,n_1)\geq t$, and therefore $\omega(S,n)\geq t$.
\end{proof}
\begin{remark} In [\cite{GH08}, Theorem 4.6]  it has been shown that in a large class of Krull monoids, the tame degree coincides with the $\tau$ invariant, whereas in numerical monoids generated by generalized arithmetic sequences it coincides with the $\omega$ invariant.
\end{remark}
\begin{remark} Let $S=\langle a,a+d,\ldots,a+xd\rangle$. In \cite{CGL} the catenary degree and the
tame degree of $S$ are found as
$$\cat(S)=\lceil\frac{a}{x}\rceil +d, \quad \tame(S)=\left\{\begin{array}{ll}
\lceil\frac{a}{x}\rceil+d & \mathrm{if}\,\, a\Mod x=1,\\
\lceil\frac{a}{x}\rceil+d+1 & \mathrm{otherwise}.\\
\end{array}\right.$$
Since
$$\lceil\frac{a-1}{x}\rceil+1=\left\{\begin{array}{ll}
\lceil\frac{a}{x}\rceil & \mathrm{if}\,\, a\Mod x=1,\\
\lceil\frac{a}{x}\rceil+1 & \mathrm{otherwise},\\
\end{array}\right.$$
our results for $h=1$ agree with these results.
\end{remark}
\begin{remark}Let $S=\langle a,ha+d,\ldots,ha+xd\rangle$ where $a,d,h$, and $x$
are positive integers, $\gcd(a,d)=1$, and $1\leq x\leq a-1$. Then
$$\tame(S)-\cat(S)=\big(\lceil\frac{a-1}{x}\rceil+1-\lceil\frac{a}{x}\rceil\big)h.$$
If $(a\mod x)\neq 1$ then $\lceil\frac{a-1}{x}\rceil+1-\lceil\frac{a}{x}\rceil=1$ and therefore, even if
the number of minimal generators is fixed,
$\tame(S)-\cat(S)$ can be arbitrary large as $h$ tends to infinity. The fact the difference of the tame degree and the catenary degree is growing holds for Krull monoids with cyclic class group or whose class group is an elementary 2-group, and every class contains a prime divisor (in both cases the difference grows with the Davenport constant). These follow from Proposition 6.5.2(2), Theorem 6.5.3(1), and Theorem 6.4.2 in \cite{GH}.
\end{remark}
\begin{remark} Let $S=\langle a,ha+d,\ldots,ha+xd\rangle$ where $a,d,h$, and $x$
are positive integers, $\gcd(a,d)=1$, and $1\leq x\leq a-1$. It is known that [\cite{OR}, Theorem 2.8]
\[g(S)=\lceil\frac{a-1}{x}\rceil ha+ad-a-d.\]
Let
\[B=\frac{g(S)+n_p}{n_1}+1=(\lceil\frac{a-1}{x}\rceil +1)h+\frac{x-1}{a}d.\]
Then
\[B-\tame(S)=\frac{x-1}{a}d\]
and
\[B-\cat(S)=(B-\tame(S))+(\tame(S)-\cat(S))=\big(\lceil\frac{a-1}{x}\rceil+1-\lceil\frac{a}{x}\rceil\big)h+
\frac{x-1}{a}d.\]
Firstly let $(a\mod x)=1$. Then $\cat(S)=\tame(S)$ while $B-\cat(S)=(x-1)d/a$ can be arbitrary large as we let $d$ goes to infinity.

Secondly let $(a\mod x)\neq 1$ and choose $d$ so that $(x-1)d<a$. Then $\lfloor B\rfloor=\tame(S)$ while
$B-\cat(S)=h+(x-1)d/a$ can be arbitrary large as we let $h$ goes to infinity.
\end{remark}
\section*{Acknowledgment}
We thank A. Geroldinger, P. A. Garc\'{i}a-S\'{a}nchez and S. T. Chapman for their help in developing this paper.

\end{document}